\documentclass[12pt,reqno]{amsart}
\usepackage{fullpage}
\usepackage{colonequals}
\usepackage{times}
\usepackage{amsmath,amssymb,amsthm,url}
\usepackage[utf8]{inputenc}
\usepackage[alphabetic]{amsrefs}
\usepackage[english]{babel}
\usepackage{comment}
\usepackage{bbm}
\usepackage{enumerate}
\usepackage{bm}
\usepackage{graphicx}
\usepackage[colorlinks=true, pdfstartview=FitH, linkcolor=blue, citecolor=blue, urlcolor=blue]{hyperref}
\usepackage{tabstackengine}
\stackMath

\newtheorem{thm}{Theorem}[section]
\newtheorem{lem}[thm]{Lemma}
\newtheorem{prop}[thm]{Proposition}
\newtheorem{defn}[thm]{Definition}
\newtheorem{ex}[thm]{Example}
\newtheorem{cor}[thm]{Corollary}

\newtheorem{conj}[thm]{Conjecture}

 \theoremstyle{definition}
\newtheorem{rem}[thm]{Remark}

\newcommand{\C}{\mathbb{C}} 
\newcommand{\N}{\mathbb{N}}
\newcommand{\Z}{\mathbb{Z}}
\newcommand{\F}{\mathbb{F}}

\AtBeginDocument{%
	\def\MR#1{}
}

\title{Van Lint--MacWilliams' conjecture and maximum cliques in Cayley graphs over finite fields}
\author{Shamil Asgarli}
\address{Department of Mathematics \\ University of British Columbia \\ 1984 Mathematics Road \\ Canada V6T 1Z2}
\email{sasgarli@math.ubc.ca}
\author{Chi Hoi Yip}
\address{Department of Mathematics \\ University of British Columbia \\ 1984 Mathematics Road \\ Canada V6T 1Z2}
\email{kyleyip@math.ubc.ca}
\date{\today}

\begin{document}

\maketitle

\begin{abstract}
A well-known conjecture due to van Lint and MacWilliams states that if $A$ is a subset of $\F_{q^2}$ such that $0,1 \in A$, $|A|=q$, and $a-b$ is a square for each $a,b \in A$, then $A$ must be the subfield $\F_q$. This conjecture is often phrased in terms of the maximum cliques in Paley graphs. It was first proved by Blokhuis and later extended by Sziklai to generalized Paley graphs. In this paper, we give a new proof of the conjecture and its variants, and show how this Erd\H{o}s-{K}o-{R}ado property of Paley graphs extends to a larger family of Cayley graphs, which we call Peisert-type graphs. These results resolve the conjectures by Mullin and Yip.
\end{abstract}

\section{Introduction}\label{sect:intro}
Throughout the paper, $p$ denotes an odd prime, $q$ denotes a positive power of $p$, and $\F_q$ denotes the finite field with $q$ elements. For a finite field $\F_q$, we denote $\F_q^*=\F_q\setminus \{0\}$. 

In 1978, van Lint and MacWilliams \cite{vLM78} studied the vectors of minimum weight in certain generalized quadratic residue codes, and they proposed the following conjecture: if $A$ is a subset of $\F_{q^2}$ such that $0,1 \in A$, $|A|=q$, and $a-b$ is a square for each $a,b \in A$, then $A$ is the subfield $\F_q$. Although this conjecture is purely a statement about the structure of the subfield, it is often phrased in terms of the maximum cliques in Paley graphs of square order, and sometimes in terms of the analog of Erd\H{o}s-{K}o-{R}ado theorem for Paley graphs \cite{GM}*{Chapter 5}.  The conjecture was first completely proved by Blokhuis \cite{Blo84}; see Section~\ref{sect:background} for a brief history of the conjecture and related results by other authors.

In this paper, we give a new proof of this conjecture (for $p$ sufficiently large with respect to $\log_p q$), explore its variant for other Cayley graphs defined on finite fields, including Peisert graphs, and confirm two related conjectures by Mullin and the second author. In addition, we exhibit several new results on maximal cliques in both generalized Paley and Peisert graphs.

We recall a few relevant terms from graph theory. Given an abelian group $G$ and a set $S \subset G \setminus \{0\}$ with $S=-S$, the {\em Cayley graph} $\operatorname{Cay}(G,S)$ is
the graph whose vertices are elements of $G$, such that two vertices $g$ and $h$ are adjacent if and only if $g-h \in S$. Such a set $S$ is called a {\em connection set}, and the condition $S=-S$ guarantees that the graph is undirected. In this paper, we introduce a new family of Cayley graphs. 

\begin{defn}[Peisert-type graphs]\label{defn:peisert-type}
Let $q$ be an odd prime power. Let $S \subset \F_{q^2}^*$ be a union of at most $\frac{q+1}{2}$ cosets of $\F_q^*$ in $\F_{q^2}^*$ such that $\F_q^* \subset S$, that is,
\begin{equation}\label{coset}
S=c_1\F_q^* \cup c_2\F_q^* \cup \cdots \cup c_m \F_q^*,   
\end{equation}
where $c_1=1$ and $m \leq \frac{q+1}{2}$. Then the Cayley graph $X=\operatorname{Cay}(\F_{q^2}^+, S)$ is said to be a Peisert-type graph.
\end{defn}

Such graphs have been implicitly studied within a larger family of strongly regular Cayley graphs in \cite{BWX99}. In Lemma~\ref{Ptlem}, we will see that Paley graphs, Peisert graphs, and more general versions of such graphs, are all special instances of Peisert-type graphs. The precise definitions of these graphs will be given in Section~\ref{sect:background}.

In order to state our first main result, we review a few additional definitions. A {\em clique} in a graph $X$ is a subset of vertices in $X$ such that every two distinct vertices in the clique are adjacent. A {\em maximum clique} is a clique with the maximum size, while a {\em maximal clique} is a clique that is not contained in a strictly larger clique. The {\em clique number} of $X$, denoted by $\omega (X)$, is the size of a maximum clique in a graph $X$. We prove the following theorem on the structure of maximum cliques in Peisert-type graphs.

\begin{thm} \label{subspacethm}
Let $X$ be a Peisert-type graph of order $q^2$, where $q$ is a power of an odd prime $p$. Then $\omega(X)= q$, and any maximum clique in $X$ containing 0 is an $\F_p$-subspace of $\F_{q^2}$. 
\end{thm}

Our second main result strengthens the conclusion of Theorem~\ref{subspacethm} under additional assumptions. 

\begin{thm}\label{subfieldthm}
Let $n \geq 2$ be an integer and $\varepsilon>0$ a real number. Let $X=\operatorname{Cay}(\F_{q^2}^+, S)$ be a Peisert-type graph, where $q=p^n$ and $p>4.1n^2/\epsilon^2$. Suppose that there is a nontrivial multiplicative character $\chi$ of $\F_{q^2}$, such that the set $\{\chi(x): x \in S\}$ is $\varepsilon$-lower bounded. Then in the Cayley graph $X$, the only maximum clique containing $0,1$ is the subfield $\F_q$.
\end{thm} 

Roughly speaking, a set $A$ is $\varepsilon$-lower bounded if the sum over any subset of $A$ is not too small. A precise definition of $\varepsilon$-lower bounded appears in Definition~\ref{defn:epsilon-bounded} in Section~\ref{sect:background}. The sharpness of Theorem~\ref{subfieldthm} is discussed in Example~\ref{ex:counter-example}. In Section~\ref{subsect:main-results}, as a straightforward application of Theorem~\ref{subfieldthm}, we will prove the van Lint--MacWilliams' conjecture for all sufficiently large $p$. 

We mention here one new application of Theorem~\ref{subfieldthm}. Mullin \cite{NM}*{Chapter 8} conjectured that every maximum clique in the Peisert graph of square order containing $0, 1$ must be the subfield $\F_q$. Note that this is an exact analogue of the van Lint--MacWilliams' conjecture for Peisert graphs. We prove Mullin's conjecture when $p$ is sufficiently large.

\begin{thm}\label{conjP*forlargep}
Let $q=p^n$ be a prime power such that $q\equiv 3 \pmod 4$. Assume that $p>8.2n^2$. Then the only maximum clique containing $0,1$ in the Peisert graph of order $q^2$ is given by the subfield $\F_q$.
\end{thm}

Another objective of the paper is to prove new results on the structure of \emph{maximal} cliques in Paley and Peisert graphs. While these results are not direct applications of our main Theorem~\ref{subfieldthm}, their proofs share similar ideas involving character sum estimates and the subspace structure. In particular, we confirm the conjecture of the second author on the maximal cliques in Peisert graphs of order $q^4$. 

\begin{thm}[\cite{Yip3}*{Conjecture 4.4}]\label{conjmaxc}
If $q$ is a power of a prime $p \equiv 3 \pmod 4$  and $q>3$, then $\F_{q}$ is a maximal clique in the Peisert graph of order $q^4$.
\end{thm}

The paper is organized as follows. In Section~\ref{sect:background}, we provide additional background, and further motivate the new definitions appearing in our paper. We also review the history of the van Lint--MacWilliams' conjecture, and put our main theorems in context. In Section~\ref{sect:tools}, we explain the two main techniques used in our proof. Specifically, we borrow results regarding the number of directions determined by a point set, and estimates on character sums over finite fields. In Section~\ref{sect:main-proofs}, we give proofs of Theorem~\ref{subspacethm} and Theorem~\ref{subfieldthm}; afterwards, we extend Theorem~\ref{conjP*forlargep} to generalized Peisert graphs. In Section~\ref{sect:maximal}, we prove Theorem~\ref{conjmaxc} and several other results on maximal cliques in certain Cayley graphs.

\section{Background and overview of the paper}\label{sect:background}

This section is organized into four subsections. In Section~\ref{subsect:history}, we give additional context for the conjecture due to van Lint and MacWilliams. In Section~\ref{subsect:peisert}, we recall the definitions of Peisert and generalized Peisert graphs, mention Mullin's conjecture, and describe what is currently known about the clique number of Peisert graphs. In Section~\ref{subsect:peisert-type}, we give further motivation for defining Peisert-type graphs. In Section~\ref{subsect:main-results}, we explain our main results in more detail, and deduce Theorem~\ref{conjP*forlargep} as a quick application.

\subsection{Van Lint--MacWilliams' conjecture and its variants}\label{subsect:history}

\begin{defn}
Let $q$ be a prime power satisfying $q \equiv 1 \pmod{4}$. The {\em Paley graph} on $\F_q$, denoted by $P_q$, is the Cayley graph $\operatorname{Cay}(\F_q^+,(\F_q^*)^2)$.
\end{defn}

In other words, $P_q$ is the graph whose vertices are the elements of $\F_q$, such that two vertices are adjacent if and only if the difference between two vertices is a square in $\F_q^{\ast}$.

Van Lint--MacWilliams' conjecture is often phrased as follows: in the Paley graph with order $q^2$, the only maximum clique containing $0,1$ is the subfield $\F_q$. Van Lint and MacWilliams \cite{vLM78} proved their conjecture in the special case $q=p$. A shorter proof for the case $q=p$ was obtained by Lov\'{a}sz and Schrijver \cite{LS83} using R\'edei's result on the number of directions; see Theorem~\ref{Ball03} for a related discussion. The conjecture was first proved completely by Blokhuis \cite{Blo84}.

\begin{thm}[\cite{Blo84}] \label{Blo84}
If $q$ is an odd prime power, then in the Paley graph of order $q^2$, the only maximum clique containing $0,1$ is the subfield $\F_{q}$. 
\end{thm}

Later, Bruen and Fisher \cite{BF91} gave a different proof using the so-called ``Jamison method". Sziklai \cite{Szi99} generalized Blokhuis's proof and extended the conjecture to certain generalized Paley graphs; see Theorem \ref{Szi99}. See Remark~\ref{differenceinproof} for discussion on the proofs given by Blokhuis and Sziklai, and how they differ from the proof in the present paper.

One can define generalized Paley graphs in an analogous way. They were first introduced by Cohen \cite{SC} in 1988, and reintroduced by Lim and Praeger \cite{LP} in 2009. 

\begin{defn}
Let $d>1$ be a positive integer. If $q \equiv 1 \pmod {2d}$, the {\em $d$-Paley graph} on $\F_q$, denoted by $GP(q,d)$, is the Cayley graph  $\operatorname{Cay}(\F_q^+,(\F_q^*)^d)$, where $(\F_q^*)^d$ is the set of $d$-th powers in $\F_q^*$.
\end{defn} 
In the literature, $3$-Paley graphs and $4$-Paley graphs are also known as {\em cubic Paley graphs} and {\em quadruple Paley graphs}, respectively.

If $d \mid (q+1)$, the following theorem of Sziklai characterizes all maximum cliques in $GP(q^2,d)$.

\begin{thm}[\cite{Szi99}*{Theorem 1.2}] \label{Szi99}
If $q$ is an odd prime power and $d$ is a divisor of $(q+1)$ such that $d>1$, then in the generalized Paley graph $GP(q^2,d)$, the only maximum clique containing $0,1$ is the subfield $\F_{q}$; in particular, each maximum clique in $GP(q^2,d)$ is an affine transformation of the subfield $\F_{q}$.
\end{thm}

\begin{rem}
The assumption that $d \mid (q+1)$ in Theorem \ref{Szi99} is crucial. In the case $d \nmid (q+1)$, it is easy to verify that $\F_{q}$ does not form a clique in $GP(q^2,d)$ (see for example \cite{Yip4}*{Lemma 1.5}); moreover, the second author \cite{Yip4} recently proved that $\omega(GP(q^2,d))\leq q-1$. Therefore, it is unlikely that there is a nice characterization of maximum cliques in $GP(q^2,d)$, where $d \nmid (q+1)$.
\end{rem}

Sziklai \cite{Szi99} also proved the following embeddability result on maximum cliques. 

\begin{thm}[\cite{Szi99}*{Theorem 1.3}] \label{embed}
Let $q$ be an odd prime power and $d$ a divisor of $(q+1)$ such that $d \geq 3$. If $C$ is a clique in the generalized Paley graph $GP(q^2,d)$, such that $0,1 \in C$, and $|C|>q-(1-1/d)\sqrt{q}$, then $C \subset \F_q$. 
\end{thm}

Many authors have attempted to give a new proof of van Lint--MacWilliams' conjecture. Godsil and Meagher \cite{GM} suggested an approach using algebraic graph theory, in particular the module method, to prove Theorem \ref{Blo84}. The module method is a powerful tool to prove Erd\H{o}s-{K}o-{R}ado type theorems; for more details, see the discussion in \cite{GM}*{Section 5.9}. Recently, Goryainov, Lin, and the authors \cite{AGLY22} have proved that all Peisert-type graphs satisfy the EKR-module property, which is the first step in applying the module method. In addition, the second author \cite{Yip4} gave a Fourier analytic characterization of maximum cliques, which suggests an alternative approach to prove Theorem \ref{Blo84}.

\subsection{Peisert graphs and generalized Peisert graphs}\label{subsect:peisert}
It is well-known that Paley graphs are self-complementary and symmetric. In \cite{WP2}, Peisert discovered a new infinite family of self-complementary symmetric graphs, which he called $P^*$-graphs. This new family of graphs also goes under the name of Peisert graphs in the literature. 

\begin{defn}
The {\em Peisert graph} of order $q=p^r$, where $p$ is a prime such that $p \equiv 3 \pmod 4$ and $r$ is even, denoted by $P^*_q$, is the Cayley graph $\operatorname{Cay}(\F_{q}^{+}, M_q)$ with $$M_q \colonequals \{g^j: j \equiv 0,1 \pmod 4\},$$ where $g$ is a primitive root of the field $\F_q$.
\end{defn}
While the construction of $P^*_q$ depends on the primitive root $g$, the isomorphism type of $P^*_q$ is independent of the choice of $g$ \cite{WP2}.  Note that $M_q$ is not closed under multiplication since $g \cdot g=g^2 \notin M_q$. The condition $p \equiv 3 \pmod 4$ is needed to ensure that $P_q^*$ is symmetric.

Kisielewicz and Peisert \cite{KP} showed that Paley graphs and Peisert graphs are similar in many aspects. However, we know very little about the structure of cliques of Peisert graphs other than Theorem \ref{KP}.

\begin{thm}[\cite{KP}*{Theorem 5.1}] \label{KP}
Let $q=p^s$, where $p \equiv 3 \pmod 4$ and $s=2k$. If $k$ is odd, then $\omega(P^*_q)= \sqrt{q}$ and the subfield $\F_{\sqrt{q}}$ forms a clique in $P^*_q$; if $k$ is even, then $\omega(P^*_q) \geq q^{1/4}$ and the subfield $\F_{q^{1/4}}$ forms a clique in $P^*_q$.
\end{thm}

Motivated by Theorem \ref{Blo84} and some numerical evidence, Mullin proposed the following conjecture \cite{NM}*{Chapter 8}. Note that it is an analog of Theorem \ref{Blo84} on the characterization of maximum cliques in a Paley graph with square order. 

\begin{conj}[Mullin]\label{conjP*}
Let $q \equiv 3 \pmod 4$ be a prime power. Then the only maximum clique containing $0,1$ in the Peisert graph of order $q^2$ is given by the subfield $\F_q$.
\end{conj}
She remarked that if $q$ is a prime, then Conjecture \ref{conjP*} can be proved using a similar technique as Theorem \ref{Szi99}; see \cite{NM}*{Corollary 3.4}. She also remarked that the proof of Theorem \ref{Szi99} fails to extend to Peisert graphs straightforwardly, since the connection set of a Peisert graph is not closed under multiplication, while for Paley graphs and generalized Paley graphs it is closed; see also Remark~\ref{differenceinproof}. We show that Conjecture~\ref{conjP*} holds for sufficiently large $q$; this is Theorem~\ref{conjP*forlargep} and its proof is given in Section~\ref{sect:main-proofs}.

In \cite{Yip4}, the second author studied the connection between maximal cliques and maximum cliques in Peisert graphs of order $q^4$, and conjectured the following based on numerical evidence. 

\begin{conj}[\cite{Yip3}*{Conjecture 4.4}]\label{conjmaxc*}
If $q$ is a power of a prime $p \equiv 3 \pmod 4$  and $q>3$, then $\F_{q}$ is a maximal clique in the Peisert graph of order $q^4$.
\end{conj}

Note that the Peisert graph of order $q^4$ is the edge-disjoint union of two copies of the quadruple Paley graph $GP(q^4,4)$. Therefore, Conjecture \ref{conjmaxc*} strengthens \cite{Yip3}*{Theorem 1.6}, which states that $\F_{q}$ is a maximal clique in $GP(q^4,4)$. We show that the conjecture above holds; see Theorem~\ref{conjmaxc}.

Motivated by the similarity shared among Paley graphs, generalized Paley graphs, and Peisert graphs, Mullin introduced generalized Peisert graphs; see \cite{NM}*{Section 5.3}. In order to ensure that $GP^*(q,d)$ is an undirected graph, we also need a further hypothesis $q \equiv 1 \pmod {2d}$; in Mullin's thesis, this last congruence condition is $q \equiv 1 \pmod d$, which appears to be a typo.

\begin{defn}
Let $d$ be a positive even integer, and $q$ a prime power such that $q \equiv 1 \pmod {2d}$. The {\em $d$-th power Peisert graph of order $q$}, denoted by $GP^*(q,d)$, is the Cayley graph $\operatorname{Cay}(\F_q^+, M_{q,d})$, where
$$
M_{q,d}=\bigg\{g^{dk+j}: 0\leq j \leq \frac{d}{2}-1, k \in \Z\bigg\},
$$
and $g$ is a primitive root of $\F_q$.
\end{defn}

\begin{rem}\label{peisert-paley-relation}
Note that $GP^*(q,2)$ is precisely the Paley graph $P_q$, and $GP^*(q,4)$ is the Peisert graph $P^*_q$ if $q$ is an even power of a prime $p \equiv 3 \pmod 4$. Moreoever, the generalized Peisert graph $GP^*(q,d)$ is the edge-disjoint union of $\frac{d}{2}$ copies of the generalized Paley graph $GP(q,d)$, which again suggests that it is much harder to characterize the maximum cliques in $GP^*(q,d)$ than doing so for $GP(q,d)$. Inspired by Conjecture~\ref{conjP*}, one can propose a similar conjecture on maximum cliques in a generalized Peisert graph; we prove such a result in Theorem~\ref{GP*thm}. 
\end{rem}

\subsection{Peisert-type graphs}\label{subsect:peisert-type}

One motivation of this paper is to provide a new and short proof for Theorem \ref{Szi99}, and explore the structure of maximum cliques of those Cayley graphs which are similar to Paley graphs and generalized Paley graphs. In particular, we consider Cayley graphs defined on $\F_{q^2}^+$ with connection set being a union of cosets of $\F_q^*$ in $\F_{q^2}^*$ as in equation \eqref{coset}. 

In our Definition~\ref{defn:peisert-type} of a Peisert-type graph $X=\operatorname{Cay}(\F_{q^2}^{+},S)$ we assume that $\F_q\subset S$; in particular, these graphs enjoy the property that $\omega(X)\geq q$. As we will see in the proof of Theorem~\ref{subspacethm}, the condition on the number of cosets $m\leq \frac{q+1}{2}$ is needed. Another reason for assuming $m\leq \frac{q+1}{2}$ is that we want $|S|\leq\frac{q^2-1}{2}$; see  Example~\ref{P*counter-example}.

We remark that our definition of Peisert-type graphs shares some similarities with the definition of Peisert graphs. Recall that a Peisert graph is simply an edge-disjoint union of two copies of the quadruple Paley graph with the same order. Note that $\operatorname{Cay}(\F_{q^2}^+, \F_q^*)=GP(q^2,q+1)$ is a $(q+1)$-Paley graph. Therefore, a Peisert-type graph is simply an edge-disjoint union of copies of $(q+1)$-Paley graph.  Notice also that we do not require the connection set $S$ to be closed under multiplication in Definition~\ref{defn:peisert-type}.

The following lemma allows us to unify the notions of generalized Paley graphs $GP(q^2,d)$ and generalized Peisert graphs $GP^*(q^2,d)$ under the assumption that $d\mid (q+1)$. Note that when $d$ is even, $GP(q,d)$ already embeds inside $GP^{\ast}(q,d)$; however, for odd $d$, there is no such obvious embedding as $GP^{\ast}(q,d)$ is defined only for even values of $d$. The lemma shows that the Peisert-type graph is general enough to encompass both cases. 

\begin{lem} \label{Ptlem}
The following families of Cayley graphs are Peisert-type graphs:
\begin{itemize}
    \item Paley graphs of square order;
    \item Peisert graph with order $q^2$, where $q \equiv 3 \pmod 4$;
    \item Generalized Paley graphs $GP(q^2,d)$, where $d \mid (q+1)$ and $d>1$;
    \item Generalized Peisert graphs $GP^*(q^2,d)$, where $d \mid (q+1)$ and $d$ is even.
\end{itemize}
\end{lem}

The proof of the preceding lemma will be presented in Section~\ref{sect:main-proofs}. Both (generalized) Paley graphs and (generalized) Peisert graphs have many nice properties, and they have many applications in coding theory and design theory; see for example \cites{JAlex15,JL,KR,WQWX}. We also expect that Peisert-type graphs will have a new set of applications.

\subsection{Main results and their implications}\label{subsect:main-results}
Our goal is to characterize maximum cliques in a Peisert-type graph $X=\operatorname{Cay}(\F_{q^2}^+,S)$. Our first main result, Theorem~\ref{subspacethm}, asserts that the clique number of $X$ is $q$, and every maximum clique in $X$ has vector space structure; see \cite{Yip3} for a related discussion on the vector space structure of maximal cliques in Cayley graphs. Later, we will show that such subspace must be in fact the subfield $\F_q$ under additional assumptions. 

When the order of the graph is $p^2$, Theorem \ref{subspacethm} immediately implies the following corollary.

\begin{cor}
Let $X$ be a Peisert-type graph of order $p^2$. Then the only maximum clique containing $0,1$ is the subfield $\F_{p}$. 
\end{cor}

It is natural to ask whether the conclusion in Theorem~\ref{subspacethm} can be strengthened to the assertion that any maximum clique is an $\F_q$-affine space. In Example~\ref{ex:counter-example} below, we will see a counterexample for this stronger conclusion.

In the same spirit as Theorem \ref{embed}, we prove an embeddability result for cliques in Peisert-type graphs.

\begin{thm} \label{stabilitythm}
Let $X=\operatorname{Cay}(\F_{q^2}^+, S)$ be a Peisert-type graph, where $q=p^n$, and $m=|S|/(q-1)<\frac{q+1}{2}$. If $C$ is a clique in $X$, such that $0 \in C$ and $|C|>q-\big(1-m/(q+1)\big)\sqrt{q}$, then $C$ is contained in a $\F_p$-subspace of $\F_{q^2}$ with dimension $n$. 
\end{thm}

We want to find minimal assumptions on Peisert-type graphs $X=\operatorname{Cay}(\F_{q^2}^+, S)$ for which the only maximum clique containing $0,1$ is the subfield $\F_q$. Such a clique is already an $\F_p$-subspace by Theorem~\ref{subspacethm}, and to show that it is exactly the subfield $\F_q$, we will use character sum estimates in Section \ref{charsumsection}. The following definition is helpful for our discussion.

\begin{defn}\label{defn:epsilon-bounded}
Let $\varepsilon>0$. A set $M \subset \C$ is said to be $\varepsilon$-lower bounded if for every integer $k \in \N$, and for every choice of $x_1,x_2,\ldots, x_k \in M$, we have
\begin{align*}
\bigg|\sum_{j=1}^{k} x_j\bigg| \geq \varepsilon k.
\end{align*}
\end{defn}

\begin{ex}\label{P*ex} \rm
If $a,b$ are non-negative integers, then 
\begin{align*}
|a+bi|=\sqrt{a^2+b^2} \geq \sqrt{(a+b)^2/2}=\frac{a+b}{\sqrt{2}}.
\end{align*}
Therefore, the set $\{1,i\}$ is $\frac{1}{\sqrt{2}}$-lower bounded. 
\end{ex}

Our Theorem~\ref{subfieldthm} has the hypothesis that $\{\chi(x): x \in S\}$ is $\varepsilon$-lower bounded for some $\varepsilon>0$. The following example explains the necessity of the condition $m\leq (q+1)/2$, or equivalently, $|S|\leq (q^2-1)/2$ in Definition~\ref{defn:peisert-type}.

\begin{ex}\label{P*counter-example} \rm
Suppose that $S\subset \F_{q^2}^{\ast}$ with $|S|>(q^2-1)/2$. If $\chi$ is an odd multiplicative character of $\F_{q^2}$, meaning that $\chi(-1)=-1$, then we claim that $S$ is not $\varepsilon$-lower bounded for \emph{any} $\varepsilon>0$. Indeed, by the pigeonhole principle, $\{b, -b\}\in S$ for some $b\in \F_{q^2}^{\ast}$, and $\chi(-b)+\chi(b)=0$, showing that the condition in Definition~\ref{defn:epsilon-bounded} is not satisfied for $k=2$. 

\end{ex}

In Lemma~\ref{Ptlem}, we saw that certain generalized Paley graphs and Peisert graphs are Peisert-type graphs. Next we present two quick applications of Theorem \ref{subfieldthm}.

It is straightforward to apply Theorem \ref{subfieldthm} to generalized Paley graph $GP(q^2,d)$, where $d \mid (q+1)$. Let $\chi$ to be a multiplicative character with order $d$; then the set $\{\chi(x^d): x \in \F_{q^2}^*\}=\{1\}$ is trivially $1$-lower bounded. This immediately implies a slightly weaker version of Theorem \ref{Szi99}, that is, we obtain a simple proof that Theorem \ref{Szi99} holds if $p$ is sufficiently large. 

\begin{rem} \label{differenceinproof}
The original proofs of Theorem~\ref{Blo84} by Blokhuis and Theorem~\ref{Szi99} by Sziklai relied on the algebraic properties of the connection sets being squares and $d$-th powers, which are in fact subgroups. Although the implication of our main result to Paley graphs is slightly weaker than their results, it seems difficult to extend their method to a general Peisert-type graph. 
\end{rem}

Similarly, we can apply Theorem \ref{subfieldthm} to Peisert graph with order $q^2$, where $q \equiv 3 \pmod 4$, and deduce that Conjecture \ref{conjP*} holds as long as $p$ is sufficiently large.

\begin{proof} [Proof of Theorem~\ref{conjP*forlargep}]
Let $g$ be a primitive root of $\F_{q^2}$ and let $\chi$ be a multiplicative character such that $\chi(g)=i$. Then the set $\{\chi(g^{4k}): k \in \N\} \cup \{\chi(g^{4k+1}): k \in \N\}=\{1,i\}$ is $\frac{1}{\sqrt{2}}$-lower bounded, as shown in Example \ref{P*ex}. The required claim follows from Theorem \ref{subfieldthm}.
\end{proof}

In Section \ref{sect:apply}, we will apply Theorem \ref{subfieldthm} to study maximum cliques in a generalized Peisert graph $GP^*(q^2,d)$, where $d \mid (q+1)$. Note that Theorem \ref{GP*thm} is a generalization of Theorem~\ref{conjP*forlargep}. It also strengthens Sziklai's result (Theorem \ref{Szi99}), because $GP(q^2,d)$ is a subgraph of $GP^{*}(q^2, d)$; see Remark~\ref{peisert-paley-relation} for more details. 

\begin{thm}\label{GP*thm}
Let $n \geq 2$ be an integer, and $d \geq 4$ an even integer. Let $X=GP^*(q^2,d)$ be a generalized Peisert graph, where $q=p^n$,  $p>4.1n^2d^4/\pi^2(d-1)^2$, and $d \mid (q+1)$. Then in $X$, the only maximum clique containing $0,1$ is the subfield $\F_q$.
\end{thm}

The following example illustrates the necessity of the assumption that $p$ is sufficiently large stated in Theorem \ref{subfieldthm}.

\begin{ex}\label{ex:counter-example}\rm
We consider the maximum cliques in a generalized Peisert graph $GP^*(q^2,q+1)$, which was shown to be a conference graph by Mullin \cite{NM}*{Lemma 5.3.6} and also a Peisert-type graph in Lemma \ref{Ptlem}.  Let $g$ be a primitive root of $\F_{q^2}^*$; then $g^{q+1}$ is a primitive root of $\F_{q}^* \subset \F_{q^2}^*$. Therefore, $GP^*(q^2,q+1)=\operatorname{Cay}(\F_{q^2}^+,S)$ is a Peisert-type graph, where
\[
S=g^0 \F_q^* \cup g^1 \F_q^* \cup \cdots \cup g^{(q-1)/2} \F_q^*.
\]
Let $C$ be a maximum clique containing $0$; then $C \setminus \{0\} \subset S$. Let $j$ be the smallest non-negative integer such that $C \cap g^j \F_q^*\neq \emptyset$, and say $ag^j \in C$, where $a \in \F_q^*$. Then $(ag^j)^{-1} C \setminus \{0\} \subset S$; in particular, $(ag^j)^{-1} C$ is a maximum clique containing $0,1$.

Suppose $\F_q$ was the only maximum clique containing $0,1$; then the above argument would show that all maximum cliques containing $0$ are given by $g^j \F_q$, where $0 \leq j \leq (q-1)/2$. This would imply that the number of maximum cliques containing $0$ is exactly $\frac{q+1}{2}$, which violates the computational results listed in \cite{NM}*{Section 5.4}. For example, $GP^*(3^4,10)$ has $9$ maximum cliques containing $0$ which is more than $\frac{q+1}{2}=\frac{9+1}{2}=5$; see \cite{AGLY22}*{Section 5.2} for a detailed study of the maximum cliques in $GP^*(3^4,10)$. Similarly, $GP^*(5^4,26)$ has $19$ maximum cliques containing $0$, which is more than $\frac{q+1}{2}=\frac{25+1}{2}=13$. This shows that Theorem~\ref{subfieldthm} can only hold for sufficiently large $p$.
\end{ex}

\section{Main tools}\label{sect:tools}

In this section, we introduce two classical tools, which are both crucial for our proof.

\subsection{Number of directions determined by a point set}
Let $AG(2,q)$ denote the {\em affine Galois plane} over the finite field $\F_q$. Let $U \subset AG(2,q)$. We use Cartesian coordinates in $AG(2,q)$ so that $U=\{(x_i,y_i):1 \leq i \leq |U|\}$.
The set of {\em directions determined by} $U \subset AG(2, q)$ is 
\[ \mathcal{D}_U=\left\{ \frac{y_j-y_i}{x_j-x_i} \colon 1\leq i <j \leq |U| \right \} \subset \F_q \cup \{\infty\},\]
where $\infty$ is the vertical direction. The theory of directions has been developed by R\'edei \cite{LR73}, Sz\H{o}nyi \cite{Sz}, and many other authors. The main methods in the theory revolve around studying certain lacunary polynomials and symmetric polynomials. In modern language, the theory was built via the application of the polynomial method over finite fields. We remark that it has been used several times to study cliques of Paley graphs and generalized Paley graphs. One highlight is that it can be used to deduce the best-known upper bounds on the clique number of Paley graphs and generalized Paley graphs; see recent papers \cites{DSW,Yip2}.

The standard way to identify $\F_{q^2}$ and $AG(2,q)$ is via an embedding with respect to a basis of $\F_{q^2}$ over $\F_q$ seen as a $2$-dimensional vector space over $\F_q$. We say that $\pi:\F_{q^2}\to AG(2,q)$ is an {\em embedding} if
\[
\pi(au+bv)=(a,b), \forall a,b \in \F_q,
\]
where $\{u,v\}$ forms a basis of $\F_{q^2}$ over $\F_q$. In the definition below, the term \textit{$F$-subspace} means a vector subspace defined over a field $F$.

\begin{defn}[\cite{BBBSS}]
Let $U$ be a subset of $AG(2,q)$ containing the origin, and let $K$ be a subfield of $\F_q$. We say $U$ is $K$-linear if there is an embedding $\pi_0$ of $\F_{q^2}$ into the $AG(2,q)$, such that $\pi_0^{-1}(U)$ forms a $K$-subspace of $\F_{q^2}$.
\end{defn}

The following theorem, due to Ball \cite{Ball03}, characterizes the number of directions determined by the graph of a function. It was built on previous works \cites{LR73,BBS,BBBSS}. We present below the simplified version of Ball's original formulation.

\begin{thm}[\cite{Ball03}]\label{Ball03}
Let $f:\F_q \to \F_q$ be any function such that $f(0)=0$, where $q$ is an odd prime power. Let $N$ be the number of directions determined by the graph of $f$. Then either $N \geq \frac{q+3}{2}$, or there is a subfield $K$ of $\F_q$ such that the graph of $f$ is $K$-linear. Moreover, in the latter case, if $K$ is the
largest subfield over which the graph is $K$-linear and $K \neq \F_q$, then $N \geq q/|K|+1$.
\end{thm}

If $U \subset AG(2,q)$ and $|U|=q$, we can try to apply Theorem \ref{Ball03} to study $\mathcal{D}_U$. This is because an affine transformation preserves the number of directions; if $U$ does not determine all the directions, then $U$ is a graph of a function after an affine transformation. The following stability result, due to Sz\H{o}nyi \cite{Sz1} and Sziklai \cite{Szi99}, enables us to extend Theorem \ref{Ball03} to any $U \subset AG(2,q)$ with size $q-O(\sqrt{q})$. 

\begin{thm}[\cite{Sz1}*{Theorem 4}, \cite{Szi99}*{Theorem 3.1}] \label{extension}
Let $U \subset AG(2,q)$ with $|U|=q-k$, where $0 \leq k \leq \alpha \sqrt{q}$, where $\alpha \in [1/2,1]$. If $U$ determines less than $(q+1)(1-\alpha)$ directions, then it can be extended to a set $U'$ with $|U'|=q$, which determines the same set of directions as $U$.
\end{thm}

\subsection{Character sums over subspaces} \label{charsumsection}
 Recall that a {\em multiplicative character} of $\F_q$ is a group homomorphism from $\F_q^*$ to the multiplicative group of complex numbers with modulus 1. It is customary to define $\chi(0)=0$. We will use $\chi_0$ to denote the trivial multiplicative character of $\F_q$. For a multiplicative character $\chi$, its order $d$ is the smallest positive integer such that $\chi^d=\chi_0$. 

Many problems in number theory and combinatorics rely on non-trivial character sum estimates over certain subsets of interest. We refer to \cite{Chang} for a nice survey paper by Chang, consisting of known results on character sums over a subset of a given finite field, such as the extension of the classical P\'olya-Vinogradov inequalities and Burgess' bounds to a general finite field. For our purpose, we need nontrivial character sum estimates over a subspace of a finite field.

Weil's estimate (see \cite{LN}*{Theorem 5.41}) is a classical tool of bounding complete character sums with a polynomial argument over a finite field. The following theorem, due to Katz \cite{Katz}, essentially generalizes Weil's estimate to character sums over a subspace. Recall an element $\theta \in \F_{q^n}$ is said to have {\em degree} $k$ over $\F_q$ if $\F_{q^k}$ is the smallest extension of $\F_q$ that contains $\theta$.

\begin{thm}[\cite{Katz}]\label{Katz}
Let $\theta$ be an element of degree $n$ over $\mathbb{F}_{q}$ and $\chi$ a non-trivial multiplicative character of $\mathbb{F}_{q^{n}} .$ Then
$$
\left|\sum_{a \in \mathbb{F}_{q}} \chi(\theta+a)\right| \leq(n-1) \sqrt{q}.
$$
\end{thm}

Recently, Reis \cite{Reis} used Theorem~\ref{Katz} to prove the following character sum estimate over an affine space.

\begin{thm}[\cite{Reis}] \label{Reis}
Let $\mathcal{A}=u+\mathcal{V} \subseteq \mathbb{F}_{q^{n}}$ be an $\mathbb{F}_{q}$-affine space of dimension $t \geq 1$, where $n>1 .$ Suppose that there exists a nonzero element $y \in \mathcal{V}$ such that the set $\mathcal{A}_{y}:=\left\{a y^{-1} \mid a \in \mathcal{A}\right\}$ contains an element of degree $n$ over $\mathbb{F}_{q} .$ If $\chi$ is a non-trivial multiplicative character of $\mathbb{F}_{q^{n}}$, then 
\begin{equation} \label{sum}
\left|\sum_{a \in \mathcal{A}} \chi(a)\right|<n \cdot q^{t-1 / 2}.    
\end{equation}
\end{thm} 

The hypothesis on the existence of a degree $n$ element cannot be completely dropped. For equation \eqref{sum} to hold, we must have some assumptions on the structure of the affine space $\mathcal{A}$. Typically, we need to assume that $\mathcal{A}$ does not have a special algebraic structure. As a counter-example, let $\mathcal{A}=\mathcal{V}=\F_{q^{t}}$ be a proper subfield. Then we can find a non-trivial multiplicative character $\chi$ on $\F_{q^n}$ such that $\chi|_{A}$ is the trivial character; in this case, 
$\sum_{a \in \mathcal{A}} \chi(a) = q^t-1>n q^{t-1/2}$ holds for sufficiently large $q$, violating the conclusion in Theorem~\ref{Reis}. This construction is an example of the so-called \emph{subfield obstruction}.  

Chang \cite{Chang}*{Section 5} also gave a somewhat weaker nontrivial estimate of character sums over subspaces of finite fields. Here we use Theorem \ref{Reis} since it is stronger to some extent, and it does not involve any implicit constants, which can be challenging to determine.

The following corollary gives a non-trivial upper bound on the character sums over a subspace when $\mathcal{A}$ is a subspace but not a subfield. Corollary \ref{charsumcor} is a consequence of Theorem \ref{Reis} and \cite{Reis}*{Proposition 2.2}; here we include the proof for the sake of completeness.

\begin{cor}\label{charsumcor}
Let $n$ be an integer such that $n \geq 2$, and $q$ an odd prime power. Let $\mathcal{V} \subseteq \mathbb{F}_{q^{{2n}}}$ be an $\mathbb{F}_{q}$-space of dimension $n$, with $1 \in \mathcal{V}$, and $\mathcal{V} \neq \mathbb{F}_{q^{n}}$. Then for any non-trivial multiplicative character $\chi$ of $\mathbb{F}_{q^{2n}}$,
\begin{equation} \label{charsum}
\left|\sum_{x \in \mathcal{V}} \chi(x)\right|< \frac{2n}{\sqrt{q}} \cdot |\mathcal{V}|.    
\end{equation}
\end{cor}
\begin{proof}

It suffices to show that if $\mathcal{V}$ has size $q^{n}$ and $\mathcal{V}\neq \F_{q^{n}}$, then $\mathcal{V}$ has an element with degree $2n$. To prove this assertion, note that the number of elements in the union of $\F_{q^{d}}$ as $d\mid 2n$ with $d<n$ is less than $q^{n-1}$. Suppose that $\mathcal{V}$ has size $q^{n}$, and $\mathcal{V}\neq \F_{q^{n}}$. If $\mathcal{V}$ did not have any element of degree $2n$, then all the elements of $\mathcal{V}\setminus \F_{q^{n}}$ would be contained in $\bigcup_{\substack{d\mid 2n, d<n}} \F_{q^d}$. Now, $\mathcal{V}\cap \F_{q^{n}}$ has dimension at most $n-1$, which means that at least $q^{n}-q^{n-1}$ many elements of $\mathcal{V}$ must belong to  $\bigcup_{\substack{d\mid 2n, d<n}} \F_{q^d}$. This is a contradiction because for $q\geq 3$, we have
\[
\bigg|\bigcup_{\substack{d\mid 2n, d<n}} \F_{q^d} \bigg| \leq \sum_{i=1}^{n-1} q^{i} = \frac{q^n-1}{q-1} - 1 \leq \frac{q^n-1}{2}-1<q^{n}-q^{n-1}.   \qedhere
\]
\end{proof}

Note that if $n$ is an odd prime, then any element in $\F_{q^n} \setminus \F_q$ has degree $n$. Thus, we obtain the following corollary of Theorem~\ref{Reis}, which will be used in Section~\ref{sect:maximal}.

\begin{cor} \label{prime}
Let $\mathcal{V} \subseteq \mathbb{F}_{q^{n}}$ be an $\mathbb{F}_{q}$-subspace of dimension $2$ , where $q$ is an odd prime power, $n$ is an odd prime with $q>n^2$, and $1 \in \mathcal{V}$.  If $\chi$ is a non-trivial multiplicative character of $\mathbb{F}_{q^{n}}$, then 
\begin{equation} \label{summ}
\left|\sum_{v \in \mathcal{V}} \chi(v)\right|<|\mathcal{V}|-1.    
\end{equation}
\end{cor} 
\begin{proof}
In view of Corollary~\ref{charsumcor}, it suffices to show that,
\begin{align*}
\frac{n}{\sqrt{q}} |\mathcal{V}| < |\mathcal{V}|-1 
\end{align*}
for $q\geq n^2+1$. Using $|\mathcal{V}|=q^2$, the above inequality is equivalent to:
\begin{align}\label{eq:corollary-prime}
\frac{n^2}{q} < \left(1-\frac{1}{q^2}\right)^2
\end{align}
The inequality~\eqref{eq:corollary-prime} follows directly from combining the two inequalities below:
\begin{align*}
    \frac{n^2}{q}\leq\frac{q-1}{q}=1-\frac{1}{q}   \ \ \ \ \text{and}  \ \ \ 1-\frac{1}{q}<\left(1-\frac{1}{q^2}\right)^2
\end{align*}
for each $q\geq 2$. \end{proof}

\section{Maximum cliques in a Peisert-type graph}\label{sect:main-proofs}

We start the section by proving Lemma \ref{Ptlem}. In Section~\ref{subsect-proof-main-results}, we prove our main theorems and mention a few corollaries. Finally, in Section~\ref{sect:apply}, we classify maximum cliques in generalized Peisert graphs as an application of our main theorem. 

\subsection{Proof of Lemma \ref{Ptlem}}\label{subsect:proof-of-Ptlem}
As a motivation for Definition~\ref{defn:peisert-type}, Lemma \ref{Ptlem} states that (generalized) Paley and Peisert graphs are special instances of Peisert-type graphs. 

\begin{proof}[Proof of Lemma \ref{Ptlem}]
The first two families of graphs are special cases of the last two. Thus, it suffices to prove the claim for generalized Paley graphs and generalized Peisert graphs. 

A generalized Paley graph $GP(q^2,d)$, where $d \mid (q+1)$ and $d>1$, is of the form $\operatorname{Cay}(\F_{q^2}^{+}, S)$ where $S$ consists of all $d$-th powers in $\F_{q^2}^{\ast}$. Since $d\mid (q+1)$, we have $\F_q^{\ast}\subset S$.  Note that $S$ is the subgroup obtained by taking the image of the group homomorphism $\F_{q^2}^{\ast}\to\F_{q^2}^{\ast}$ sending $x\mapsto x^d$. It follows that $|S|=\frac{q^2-1}{d}$. Finally, $\F_{q}^{\ast}$ is a subgroup of $S$ with index
$$
\frac{|S|}{|\F_q^{\ast}|} = \frac{(q^2-1)/d}{q-1}=\frac{q+1}{d}
$$
which allows us to write,
$$
S=c_1 \F_q^{\ast} \cup \cdots \cup c_m \F_q^{\ast}
$$
with $c_1=1$ and $m=\frac{q+1}{d}\leq \frac{q+1}{2}$. Thus, $GP(q^2,d)$ is a Peisert-type graph.

For the generalized Peisert graph $GP^*(q^2,d)$, where $d \mid (q+1)$ and $d$ is even, the connection set $S$ can be expressed as a disjoint union:
$$
S = \bigsqcup_{i=0}^{d/2-1} g^{i} (\F_{q^2}^{\ast})^d. 
$$
It follows that $GP^*(q^2,d)$ is the edge-disjoint union of $\frac{d}{2}$ copies of $GP(q^2, d)$. Since each $g^{i} (\F_{q^2}^{\ast})^d$ can be expressed as a union of $\frac{q+1}{d}$ cosets of $\F_q^{\ast}$, we can express $S$ above as,
$$
S=c_1 \F_q^{\ast} \cup \cdots \cup c_m \F_q^{\ast}
$$
with $c_1=1$ and $m=\frac{q+1}{d} \cdot \frac{d}{2}=\frac{q+1}{2}$. Thus, $GP^*(q^2,d)$ is a Peisert-type graph.
\end{proof}

\subsection{Proof of main results and consequences}\label{subsect-proof-main-results}
With the two tools described in the previous section, we are ready to prove Theorem \ref{subspacethm}. 

\begin{proof}[Proof of Theorem \ref{subspacethm}]
Let $v \in \F_{q^2}^* \setminus S$. Since $\F_q^* \subset S$, $\{1,v\}$ forms a basis of $\F_{q^2}$ over $\F_q$. We consider the following standard embedding $\pi$ of $\F_{q^2}$ into $AG(2,q)$, where 
\[
\pi(a+bv)=(a,b), \quad \forall a,b \in \F_q.
\]
It is clear that the subfield $\F_q$ forms a clique in the Cayley graph $X$, so $\omega(X)\geq q$. 

Let $C$ be a maximum clique in $X$ containing $0$. We can express $U$, the image of $C$ under the embedding $\pi$ as 
\[
U=\pi(C)=\{(a_j,b_j): 1 \leq j \leq |C|\} \subset AG(2,q).
\]
If there are $1 \leq j<k \leq |C|$ such that $a_j=a_k$, then $b_j \neq b_k$, and $(b_j-b_k)v \in S \cap v\F_q^*$ since $C$ is a clique. Since $S$ is a union of cosets of $\F_q^*$, this implies that $v \in S$, contradicting our assumption. This also rules out the possibility that $|C| \geq q+1$ by the pigeonhole principle. Therefore, $|C|=\omega(X)=q$. 

Let $S=\bigcup_{j=1}^{m} c_j \F_q^*$. Note that $C=\{a_j+b_jv: 1\leq j \leq q\}$. As we saw above, for any $1 \leq j<k \leq q$, we have $a_j \neq a_k$. Since $C$ is a clique, we have $(a_j+b_jv)-(a_k+b_kv)=(a_j-a_k)+(b_j-b_k)v \in S$; it follows that 
\begin{align*}
1+\frac{b_j-b_k}{a_j-a_k} v \in S \cap (1+\F_q v),
\end{align*}
since $S$ is a union of cosets of $\F_q^*$, $a_j-a_k \in \F_q^*$, and $b_j-b_k \in \F_q$.

Notice that 
\[
S \cap (1+\F_q v)= \bigcup_{j=1}^{m} \big(c_j \F_q^*\cap (1+\F_q v)\big).
\]
We claim that $|c_j \F_q^*\cap (1+\F_q v)| \leq 1$ for each $1 \leq j \leq m$. Indeed, if $x, y\in c_j \F_q^*\cap (1+\F_q v)$, then $x-y= t v$ for some $t \in \F_q$; combining $t v \in c_j \F_q$ with $v\notin S$, we get $t=0$, and so $x=y$.

As we saw above, $U$ does not determine the vertical direction. Since $|U|=q$, we must have that
\begin{align*}
U=\{(x,f(x)): x \in \F_q\}
\end{align*}
is the graph of some function $f:\F_q \to \F_q$ such that $f(0)=0$. Recall that the direction set $\mathcal{D}_U$ consists of all ratios $\frac{b_j-b_k}{a_j-a_k}$ for $1\leq j < k \leq q$. We have shown that, 
\[
1+\mathcal{D}_U \cdot v \subset  S \cap (1+\F_q v)= \bigcup_{j=1}^{m} \big(c_j \F_q^*\cap (1+\F_q v)\big) 
\]
has size at most $m \leq \frac{q+1}{2}$. Equivalently, $|\mathcal{D}_U|\leq \frac{q+1}{2}$. By Theorem \ref{Ball03}, it follows that $U$ is $K$-linear for some subfield $K\subset \F_q$. In particular, $U$ is $\F_p$-linear, and so $C$ is an $\F_p$-subspace of $\F_{q^2}$.
\end{proof}

Inspired by the proof of Theorem~\ref{subspacethm}, we provide a sufficient condition for a maximum clique in a Peisert-type graph on $\F_{q^2}$ to be the subfield $\F_q$. 

\begin{cor} \label{cor:improvement}
Let $X=\operatorname{Cay}(\F_{q^2}^+, S)$ be a Peisert-type graph, where $q=p^n$. Suppose that 
$m=|S|/(q-1)\leq p^{n-k}$, where $k$ is the largest proper divisor of $n$. Then the only maximum clique in $X$ containing $0,1$ is the subfield $\F_q$. 
\end{cor}

\begin{proof}
Let $C$ be a maximum clique in $X$ such that $0\in C$. Following the proof of Theorem \ref{subspacethm}, if we embed $C$ into $AG(2,q)$, then its image $U$ determines at most $m$ directions and $U$ is $K$-linear, where $K$ is a subfield of $\F_q$. If $K \neq \F_q$, then Theorem \ref{Ball03} implies that $m \geq q/|K|+1 \geq p^{n-k}+1$, contradicting the assumption. Therefore, $U$ is $\F_q$-linear, that is, $C$ is an $\F_q$-subspace. In particular, if $1 \in C$, then $C$ must be the subfield $\F_q$.
\end{proof}

\begin{rem}
There are counterexamples of Peisert-type graphs with $m=\frac{p^n-1}{p^k-1}$ where $\F_q$ is \emph{not} the only maximum clique containing $0$ and $1$. Explicit constructions can be found in \cite{AGLY22}*{Section 5.1}. Note that $\frac{p^n-1}{p^k-1}$ is only slightly bigger than $p^{n-k}$, which means that the previous corollary is nearly sharp. 
\end{rem}

With the help of Theorem~\ref{extension}, namely the stability result on the direction set, the proof of Theorem~\ref{subspacethm} can be easily modified to prove Theorem~\ref{stabilitythm}.

\begin{proof}[Proof of Theorem~\ref{stabilitythm}]
Let $C$ be a clique in $X$, such that $0 \in C$ and $|C|>q-\big(1-m/(q+1)\big)\sqrt{q}$. Similar to the proof of Theorem~\ref{subspacethm}, we can embed $C$ into $AG(2,q)$ using the same map $\pi$ to obtain $U \subset AG(2,q)$. By the exact same reasoning, $U$ determines at most $m<\frac{q+1}{2}$ directions, and $U$ does not determine the vertical direction. 

We can now apply Theorem~\ref{extension} with $|U|=q-k$ where $k<\left(1-\frac{m}{q+1}\right)\sqrt{q}$ and $\alpha=1-\frac{m}{q+1}-\varepsilon\in [1/2, 1]$ for a sufficiently small $\varepsilon>0$. This allows us to extend $U$ to a larger set $U' \subset AG(2,q)$, such that $|U'|=q$, and $U'$ determines the same set of directions as $U$; in particular, $1+\mathcal{D}_{U}v=1+\mathcal{D}_{U'}v$. Let $(a,b)\neq (a',b') \in U'$; then $a \neq a'$, and 
\[
1+\frac{b'-b}{a'-a} v \in 1+\mathcal{D}_{U'}v \subset S.
\]
It follows that $(a'-a)+(b'-b)v \in S$; in other words, the preimage $C'=\pi^{-1}(U')$ forms a maximum clique in $X$. The required claim follows from Theorem \ref{subspacethm}.
\end{proof}

Now we are ready to present the proof of Theorem \ref{subfieldthm}.

\begin{proof}[Proof of Theorem \ref{subfieldthm}]
We can always assume $\varepsilon\leq 1$ in the statement of the theorem since characters have absolute value at most $1$. Suppose that there is a nontrivial multiplicative character $\chi$ of $\F_{q^2}$ such that the set $\{\chi(x): x \in S\}$ is $\varepsilon$-lower bounded. 
By Theorem \ref{subspacethm}, if $0\in C$ is a clique in the Peisert-type graph $X$, with $|C|=q$, then $C$ is an $\F_p$-subspace of $\F_{q^2}$. 

Suppose $0,1 \in C$, and $C \neq \F_q$. Then $C$ is an $\F_p$-subspace of $\F_{q^2}=\F_{p^{2n}}$ and $C\setminus \{0\} \subset S$. By Corollary \ref{charsumcor}, 
\[
\left|\sum_{x \in C} \chi(x)\right|< \frac{2n}{\sqrt{p}} \cdot |C|.    
\]
On the other hand, the set $\{\chi(x): x \in S\}$ is $\varepsilon$-lower bounded, which guarantees that
$$
\left|\sum_{x \in C} \chi(x)\right| \geq \varepsilon(|C|-1).
$$
Comparing these two inequalities, we obtain
$$
p < \frac{4n^2}{\varepsilon^2} \left(\frac{q}{q-1}\right)^2. 
$$
Now, $p>4n^2\geq 16$ and thus $q\geq p^2 \geq 17^2=289$. Thus, 
\begin{align*}
p<\left(\frac{4n^2}{\varepsilon^2}\right)\left(\frac{289}{288}\right)^2\approx \frac{4.02783n^2}{\varepsilon^2}
\end{align*} which contradicts our hypothesis. Therefore, if $0,1 \in C$, then $C=\F_q$. 
\end{proof}

\begin{rem}
Under additional conditions on the size of the connection set, the range for the values of $p$ in Theorem~\ref{subfieldthm} can be improved. We already saw this phenomenon in Corollary~\ref{cor:improvement}. More generally, suppose that $\frac{|S|}{q-1}\leq q^{1-\delta}$ where $\delta=1/s$ and $s$ is a proper divisor of $n$. Then using Theorem~\ref{Ball03} and applying Theorem~\ref{Reis} over the base field $\F_{q^{\delta}}=\F_{p^{n/s}}$, we see that Theorem~\ref{subfieldthm} will hold under the weaker hypothesis $p^{n/s} > \frac{9s^2}{\varepsilon^2}$.
\end{rem}

The following corollary follows immediately from Theorem~\ref{subfieldthm} and Theorem~\ref{stabilitythm}.

\begin{cor}
Under the same assumptions as in Theorem~\ref{subfieldthm}, if $|S|<\frac{q^2-1}{2}$, then for any clique $C$ in $X$ with $|C|>q-\big(1-\frac{m}{q+1}\big)\sqrt{q}$ and $0,1 \in C$, we have $C \subset \F_q$.
\end{cor}

\subsection{Application to generalized Peisert graphs}\label{sect:apply}

We start with the following lemma, concerning the connection set of a generalized Peisert graph.

\begin{lem}\label{GP*lem} 
Let $d \geq 4$ be an even integer, and $\theta=\exp(2\pi i/d)$. Then the set $M=\{\theta^j :0 \leq j \leq d/2-1\}$ is $\left(\frac{\pi}{d}-\frac{\pi}{d^2}\right)$-lower bounded.
\end{lem}
\begin{proof}
Let $a_j$ be a non-negative integer for each $0\leq j\leq d/2-1$, and $b=\sum_{j=1}^{d/2-1} a_j$. We have
\begin{align*}
\Im \bigg(\sum_{j=0}^{d/2-1} a_j \theta^j\bigg) =\sum_{j=1}^{d/2-1} a_j \sin\left (\frac{2\pi j}{d}\right) \geq \sum_{j=1}^{d/2-1} a_j \sin\left (\frac{2\pi}{d}\right) = b\sin \frac{2\pi}{d}.    
\end{align*}
If $a_0 \leq b \cos \frac{2\pi}{d}$, then 
\begin{align*}
\bigg|\sum_{j=0}^{d/2-1} a_j \theta^j\bigg| \geq \Im \bigg(\sum_{j=0}^{d/2-1} a_j \theta^j\bigg) \geq b\sin \frac{2\pi}{d} = b\left(\frac{1+\cos \frac{2\pi}{d}}{1+\cos \frac{2\pi}{d}}\right)\sin\left(\frac{2\pi}{d}\right)  \geq (b+a_0) \frac{\sin \frac{2\pi}{d}}{1+\cos \frac{2\pi}{d}}.
\end{align*}
If $a_0> b \cos \frac{2\pi}{d}$, then 
\begin{align*}
\Re \bigg(\sum_{j=0}^{d/2-1} a_j \theta^j\bigg) =a_0+\sum_{j=1}^{d/2-1} a_j \cos\left (\frac{2\pi j}{d}\right) \geq a_0-\sum_{j=1}^{d/2-1} a_j \cos\left (\frac{2\pi}{d}\right) = a_0-b\cos \frac{2\pi}{d}> 0;
\end{align*}
therefore, 
\[
\bigg|\sum_{j=0}^{d/2-1} a_j \theta^j\bigg|^2 \geq
\bigg(a_0-b \cos \frac{2\pi}{d}\bigg)^2+\bigg(b\sin \frac{2\pi}{d}\bigg)^2=(a_0+b)^2-2a_0b \bigg(1+\cos \frac{2\pi}{d}\bigg) \geq (a_0+b)^2 \frac{1-\cos \frac{2\pi}{d}}{2}.
\]
Note that we always have
\[
\frac{\sin \frac{2\pi}{d}}{1+\cos \frac{2\pi}{d}} \geq \sqrt{\frac{1-\cos \frac{2\pi}{d}}{2}}.
\]
Therefore, $M$ is $\sqrt{\frac{1-\cos \frac{2\pi}{d}}{2}}$-lower bounded. As $d \to \infty$, $M$ is $\frac{\pi}{d}+O(d^{-2})$-lower bounded. In particular, one can check that $\sqrt{\frac{1-\cos \frac{2\pi}{d}}{2}}\geq \frac{\pi}{d}-\frac{\pi}{d^2}$ using Taylor series expansion, and thus, $M$ is $\left(\frac{\pi}{d}-\frac{\pi}{d^2}\right)$-lower bounded.
\end{proof}

Now we are ready to prove Theorem \ref{GP*thm}.

\begin{proof}[Proof of Theorem \ref{GP*thm}]
Let $g$ be a primitive root of $\F_{q^2}^*$, and $\theta=\exp(2\pi i/d)$. Let $\chi$ be the multiplicative character in $\F_{q^2}$ such that $\chi(g)=\theta$. Recall that the connection set $S$ of the generalized Peisert graph $X=GP^*(q^2,d)$ is given by
$$
S=\bigg\{g^{dk+j}: 0\leq j \leq \frac{d}{2}-1, k \in \Z\bigg\}.
$$
Therefore, $\{\chi(x):x \in S\}=\{\theta^j :0 \leq j \leq d/2-1\}$ is $\left(\frac{\pi}{d}-\frac{\pi}{d^2}\right)$-lower bounded by Lemma \ref{GP*lem}. We have shown that $X$ is a Peisert-type graph in Lemma \ref{Ptlem}; therefore, we can apply Theorem \ref{subfieldthm} to get the desired characterization of maximum cliques in $X$ provided that
\[
p>\frac{4.1n^2}{\big(\frac{\pi}{d}-\frac{\pi}{d^2}\big)^2}=\frac{4.1n^2d^4}{\pi^2(d-1)^2}.\qedhere
\]
\end{proof}

\section{Maximal cliques in generalized Paley graphs and Peisert graphs}\label{sect:maximal}

Recall that our strategy to classify maximum cliques is the following: first show that every maximum clique has a subspace structure, and then apply character sum estimates over a subspace to show a maximum clique must be a subfield. In this section, we employ a similar idea to study maximal cliques. Assuming that a given clique is not maximal, if we can extend the clique and construct a maximal clique with a subspace structure, then it is possible that we can apply character sum estimates to derive a contradiction. 

For generalized Paley graphs, this approach is made possible via the following lemma, established by the second author \cite{Yip3}.

\begin{lem}[\cite{Yip3}*{Corollary 3.1}] \label{maximal}
Let $GP(q,d)$ be a generalized Paley graph. If $K$ is a proper subfield such that $K$ forms a clique in $GP(q,d)$, then there is a $K$-subspace $\mathcal{V}$ of $\F_q$, such that $K \subset \mathcal{V}$, and $\mathcal{V}$ forms a maximal clique in $GP(q,d)$. 
\end{lem}

Our aim is to provide partial progress towards the following main conjecture in \cite{Yip3}:
\begin{conj}[\cite{Yip3}*{Conjecture 1.4}]\label{maxconj}
Let $d$ be a positive integer greater than $1$. Let $q \equiv 1 \pmod {2d}$ be a power of a prime $p$, and let $r$ be the largest integer such that $\F_{p^r}$ is a subfield of $\F_q$ and $d \mid \frac{q-1}{p^r-1}$. Then the subfield $\F_{p^r}$ forms a maximal clique in $GP(q,d)$.
\end{conj}

Conjecture \ref{maxconj} has been shown to be true for cubic Paley graphs with cubic order and quadruple Paley graphs with quartic order in \cite{Yip3}. We establish the following special case of Conjecture \ref{maxconj}.
\begin{thm}\label{maximalGP}
Let $d$ be a positive integer greater than $1$. Let $n$ be an odd prime. If $q>n^2$ and $d \mid \frac{q^n-1}{q-1}$, then the subfield $\F_q$ forms a maximal clique in $GP(q^n,d)$.
\end{thm}

\begin{proof}
Since $q$ is an odd prime power,  $d \mid \frac{q^n-1}{q-1}$ implies that $q^n \equiv 1 \pmod {2d}$, and thus $GP(q^n,d)$ is a well-defined $d$-Paley graph. It is easy to verify that the condition $d \mid \frac{q^n-1}{q-1}$ guarantees that $\F_q$ forms a clique in $GP(q^n,d)$; see \cite{BDR}*{Theorem 1}.

Suppose that $\F_q$ is not a maximal clique in $GP(q^n,d)$. Then Lemma \ref{maximal} implies that there is a clique $\mathcal{V}$ in $GP(q^n,d)$, such that $\mathcal{V}$ is a 2-dimensional $\F_q$-subspace, and $\F_q \subset \mathcal{V}$. Therefore, if we let $\chi$ to be a multiplicative character of $\F_{q^n}$ with order $d$, then 
\begin{equation} 
\left|\sum_{v \in \mathcal{V}} \chi(v)\right|=|\mathcal{V}|-1,    
\end{equation}
which violates inequality \eqref{summ} in Corollary \ref{prime}, provided $q>n^2$. This shows that $\F_q$ is a maximal clique in $GP(q^n,d)$.
\end{proof}

To prove Theorem \ref{conjmaxc}, we use a similar strategy. We need the following result from \cite{Yip3}, which explores the connection between maximal cliques and maximum cliques in a Peisert graph of order $q^4$.

\begin{thm}[\cite{Yip3}*{Theorem 1.7}]\label{maxP*}
Let $q$ be a power of a prime $p \equiv 3 \pmod 4$.  If $\F_{q}$ is not a maximal clique in the Peisert graph of order $q^4$, then $\omega(P^*_{q^4})=q^2$; moreover, there exists $h \in \F_{q^4} \setminus \F_{q}$, such that $\F_{q} \oplus h\F_{q}$ forms a maximum clique. 
\end{thm}

We now have the tools to prove Theorem \ref{conjmaxc}.

\begin{proof}[Proof of Theorem~\ref{conjmaxc}]
Let $g$ be a primitive root of $\F_{q^4}$. Let $\chi$ be the multiplicative character of $\F_{q^4}$ such that $\chi(g)=i$. It follows that for any $k \in \N$, $\chi(g^{4k})=1$, and $\chi(g^{4k+1})=i$. 

Suppose that $\F_{q}$ is not a maximal clique in the Peisert graph $P^*_{q^4}$. By Theorem \ref{maxP*}, there exists $h \in \F_{q^4} \setminus \F_{q}$, such that $\mathcal{V}=\F_{q} \oplus h\F_{q}$ forms a maximum clique. In particular, for each $x \in \mathcal{V}\setminus \{0\}$, we have $\chi(x)=i$ or $\chi(x)=1$.
Consequently,
\[
\sum_{x \in \mathcal{V}} \chi(x)=a+bi,
\]
where $a,b$ are non-negative integers such that $a+b=|\mathcal{V}|-1$. Therefore,
\begin{equation}\label{eqq1}
\bigg|\sum_{x \in \mathcal{V}} \chi(x)\bigg|=\sqrt{a^2+b^2} \geq \sqrt{(a+b)^2/2}=\frac{|\mathcal{V}|-1}{\sqrt{2}}=\frac{q^2-1}{\sqrt{2}}.    
\end{equation}
On the other hand, $\mathcal{V}$ is an  $\F_{q}$-subspace with dimension $2$. Moreover, note that $g^{q^2+1}$ is a primitive root of the subfield $\F_{q^2}$; in particular, since $q^2+1 \equiv 2 \pmod 4$, the subfield $\F_{q^2}$ does not form a clique in the Peisert graph of order $q^4$. Therefore, $\mathcal{V} \neq \F_{q^2}$. By Corollary \ref{charsumcor},
\begin{equation}\label{eqq2}
\bigg|\sum_{x \in \mathcal{V}} \chi(x)\bigg| < 4q^{3/2}.
\end{equation}
Combining inequalities \eqref{eqq1} and \eqref{eqq2}, we must have $q<33$. This implies that when $q \geq 33$, $\F_{q}$ is a maximal clique in the Peisert graph of order $q^4$. 

Finally, in \cite{Yip3}, using SageMath, we have verified that for each $q \in \{7,9,11,19,23,27,31\}$, $\F_{q}$ is a maximal clique in the Peisert graph of order $q^4$. 
\end{proof}

\begin{rem}
The assumption $q>3$ is necessary in Theorem~\ref{conjmaxc}. Indeed, one can check that when $q=3$, the subfield $\F_3$ is not maximal in the Peisert graph of order $81$. Indeed, there is a clique which is a two-dimensional $\F_3$-space containing $\F_3$. 
\end{rem}

We conclude the section with a similar statement regarding maximal cliques in a generalized Peisert graph.

\begin{thm}
Let $d \geq 4$ be a positive even integer. Let $n$ be an odd prime, and $q$ be an even power of an odd prime $p$. If $q>1.03n^2d^4/\pi^2(d-1)^2$ and $d \mid \frac{q^n-1}{q-1}$, then the subfield $\F_q$ forms a maximal clique in $GP^*(q^n,d)$.
\end{thm}

The proof is similar to the proof of Theorem \ref{GP*thm} and Theorem \ref{maximalGP}, except that we need to establish the following variant of Theorem \ref{maxP*} for $GP^*(q^n,d)$ (which is straightforward by following the proof of Theorem \ref{maxP*} step by step), and use Lemma \ref{GP*lem} to get a lower bound on the character sum estimate.

\begin{prop}
Let $d \geq 4$ be a positive even integer. Let $n$ be an odd prime, and $q$ be an even power of an odd prime $p$. If $d \mid \frac{q^n-1}{q-1}$, and if the subfield $\F_q$ is not a maximal clique in $GP^*(q^n,d)$, then  there exists $h \in \F_{q^n} \setminus \F_q$, such that $\F_q \oplus h\F_q$ forms a clique in $X$. 
\end{prop}

\section*{Acknowledgements}
The authors thank Greg Martin, Daniel Panario, Zinovy Reichstein, J\'ozsef Solymosi, and Joshua Zahl for helpful discussions. The authors are also grateful to the referees for valuable comments, corrections, and suggestions. The research of the first author was supported by a postdoctoral research fellowship from the University of British Columbia. The research of the second author was supported in part by a Four Year Doctoral Fellowship from the University of British Columbia.

\bibliographystyle{alpha}
\bibliography{biblio.bib}

\end{document}